%% file: magulfarXiv.tex
\documentclass[12pt,reqno]{amsart}
\usepackage[latin1]{inputenc}
\usepackage{subfigure}

\usepackage{amsmath}
\usepackage{amssymb}
\usepackage{ifthen}
\usepackage{xargs}
\usepackage{booktabs}
\usepackage{paralist}

\usepackage{tikz}
\usetikzlibrary{patterns}

\usepackage[vlined,linesnumbered,norelsize]{algorithm2e}

\usepackage{stmaryrd}

\input{patternmacros}

\newcommand{\mmpatterna}[5]{
  \raisebox{0.6ex}{
  \begin{tikzpicture}[scale=0.35, baseline=(current bounding box.center), #1]
    \useasboundingbox (0.0,-0.1) rectangle (#2+1.4,#2+1.1);
  
    \shadetheboxes{#4}
    
    \drawthegrid{#2}
    
    \fill[color = white!100, opacity=1, rounded corners = 1.5pt] (2,3.9) -- (1.1,3.9) -- (1.1,2.1) -- (2.1,2.1) -- (2.1,1.1) -- (3.9,1.1) -- (3.9,2.9) -- (2.9,2.9) -- (2.9,3.9) -- (2,3.9);
    \draw[color = black, rounded corners = 1.5pt] (2,3.9) -- (1.1,3.9) -- (1.1,2.1) -- (2.1,2.1) -- (2.1,1.1) -- (3.9,1.1) -- (3.9,2.9) -- (2.9,2.9) -- (2.9,3.9) -- (2,3.9);
    
    \fill[black] (2.5,2.5) node {$\scriptstyle1$};
    
    \drawspecialbox{#5}
    
    \drawtheclpattern{#3}
    
  \end{tikzpicture}}
}

\newcommand{\mmpatternb}[5]{
  \raisebox{0.6ex}{
  \begin{tikzpicture}[scale=0.35, baseline=(current bounding box.center), #1]
    \useasboundingbox (0.0,-0.1) rectangle (#2+1.4,#2+1.1);
  
    \shadetheboxes{#4}
    
    \drawthegrid{#2}
       
    \fill[color = white!100, opacity=1, rounded corners = 1.5pt] (4,4.9) -- (3.1,4.9) -- (3.1,4.1) -- (4.1,4.1) -- (4.1,3.1) -- (4.9,3.1) -- (4.9,4.9) -- (4,4.9);
    \draw[color = black, rounded corners = 1.5pt] (4,4.9) -- (3.1,4.9) -- (3.1,4.1) -- (4.1,4.1) -- (4.1,3.1) -- (4.9,3.1) -- (4.9,4.9) -- (4,4.9);
    
    \fill[color = white!100, opacity=1, rounded corners = 1.5pt] (0.1,1) -- (0.1,1.9) -- (0.9,1.9) -- (0.9,0.9) -- (1.9,0.9) -- (1.9,0.1) -- (0.1,0.1) -- (0.1,1);
    \draw[color = black, rounded corners = 1.5pt] (0.1,1) -- (0.1,1.9) -- (0.9,1.9) -- (0.9,0.9) -- (1.9,0.9) -- (1.9,0.1) -- (0.1,0.1) -- (0.1,1);
        
    \drawspecialbox{#5}
    
    \drawtheclpattern{#3}
    
  \end{tikzpicture}}
}

\newcommand{\Av}{\mathrm{Av}}

\newcommand{\subwords}{\mathrm{subwords}}
\newcommand{\fl}{\mathrm{fl}}
\newcommand{\sh}{\mathrm{sh}}
\newcommand{\forb}{\mathrm{forb}}
\newcommand{\inv}{\mathrm{inv}}

\newcommand{\stein}{\mathsf{Forb}}
\newcommand{\mine}{\mathsf{Mine}}
\newcommand{\bisc}{\mathsf{BiSC}}

\newcommand{\dbrac}[1]{{\llbracket #1 \rrbracket}}  

\newcommand{\cand}{\mathsf{cand}}

\pgfmathsetmacro{\patttablecale}{1.05}
\pgfmathsetmacro{\pattdispscale}{0.80}
\pgfmathsetmacro{\patttextscale}{0.6}

\newtheorem{theorem}{Theorem}[section]
\newtheorem{proposition}[theorem]{Proposition}
\newtheorem{lemma}[theorem]{Lemma}

\usepackage[pdftex]{hyperref}

\begin{document}

\title[Algorithms for discovering and proving theorems]{Algorithms for discovering and proving theorems
about permutation patterns}

\author{Hjalti Magnusson \and Henning Ulfarsson}

\address{School of Computer Science, Reykjav\'ik University, Menntavegi 1, 101 Reykjav\'ik, Iceland}

\email{hjaltim07@ru.is, henningu@ru.is}

\thanks{The authors are supported by grant no.\ 090038013-4 from the Icelandic Research Fund.}

\begin{abstract}
We present an algorithm, called BiSC, that describes the patterns avoided by a given set of permutations. It automatically conjectures the statements of known theorems such as the descriptions of stack-sortable (Knuth 1975) and West-2-stack-sortable permutations (West 1990), smooth (Lakshmibai and Sandhya 1990) and forest-like permutations (Bousquet-M{\'e}lou and Butler 2007), and simsun permutations (Br{\"a}nd{\'e}n and Claesson 2011). The algorithm has also been used to discover new theorems and conjectures related to Young tableaux, Wilf-equivalences and sorting devices. We further give algorithms to prove a complete description of preimages of pattern classes under certain sorting devices. These generalize an algorithm of Claesson and Ulfarsson (2012) and allow us to prove a linear time algorithm for finding occurrences of the pattern 4312.
\end{abstract}


\maketitle


\setcounter{tocdepth}{1}
\tableofcontents

\section{Introduction}
\label{sec:in}

A \emph{permutation} of length $n$ is a bijection from the set $\{1, \dotsc, n\}$ to itself. We write
permutations in the \emph{one-line notation}, where $\pi_1 \pi_2 \dotsm \pi_n$ is
the permutation that sends $i$ to $\pi_i$. If $w = w_1 w_2 \dotsm w_k$ is a word of distinct integers,
$\fl(w)$ is the permutation obtained by replacing the $i$th smallest letter in $w$ with $i$. This is
called the \emph{flattening} of $w$.
A permutation $\pi$ of length $n$ \emph{contains} a permutation $p$ of length $k$ if there exist indices
$1 \leq j_1 < j_2 < \dotsm < j_k \leq n$ such that $\fl(\pi_{j_1} \pi_{j_2} \dotsm \pi_{j_k}) = p$. In this context $p$ is called a \emph{(classical) pattern}. If $\pi$ does not contain $p$
then it \emph{avoids} $p$. We let $\Av(P)$ denote the set of permutations that avoid all the patterns in
a set $P$.

This extended abstract summarizes two papers, \cite{U} and \cite{HU}, which treat algorithms and
permutation patterns.
The first paper introduces an algorithm, $\bisc$,  which was inspired by a question posed
by Billey~\cite{SaraTalk}. Its input is a set of permutations and its output
is a set of mesh patterns that the permutations avoid. Mesh patterns are a type of generalized
pattern introduced by Br{\"a}nd{\'e}n and Claesson~\cite{BC}.
$\bisc$ can rediscover the statements of well-known
theorems describing properties of permutations with patterns. A few
examples may be seen in Tab.~\ref{tab:props}.

\begin{table}[htdp]
\begin{center}
\begin{tabular}{lll}
\toprule
Property                              & Forbidden patterns               & Reference \\
\midrule
stack-sortable                    & $231$                                    & \cite[Section 2.2.1, Exc.~4]{K} \\
West-$2$-stack-sortable & $2341$, $(3241,\{(1,4)\})$ & \cite[Thm.~4.2.18]{W90} \\
simsun                               & $(321,\{(1,0),(1,1),(2,2)\})$  & \cite[p.~7]{BC} \\
smooth                                & $1324$, $2143$                  & \cite[Thm.~1]{LS} \\
forest-like                            & $1324$, $(2143,\{(2,2)\})$  & \cite[Thm.~1]{BMB} \\
\bottomrule
\end{tabular}
\end{center}
\caption{Some statements of known theorems $\bisc$ can rediscover}
\label{tab:props}
\end{table}%

\noindent
We emphasize that $\bisc$ discovers statements like the ones in
Tab.~\ref{tab:props}, but does not supply a proof. In subsection~\ref{subsec:appl} we present
new theorems found by us and others using $\bisc$.

Section~\ref{sec:sortingalgo} summarizes the second paper~\cite{HU}, on algorithms that
prove a complete description of preimages of classical pattern classes under a sorting operator.
These are based on algorithms in~\cite{CU} for the stack-sorting and bubble-sorting operator.
We give a common generalization to a stack of any depth, as well as algorithms for sorting with a queue,
a pop stack, with insertion and with the pancake sorting method. We only describe the first two in this extended abstract.
We give a linear time (in the size of the input permutation) algorithm for recognizing the
pattern $4312$. This algorithm was discovered with $\bisc$
and proven with the preimage algorithms for a queue and a stack.
It is the first linear algorithm for recognizing a pattern of length greater than $3$ (besides the increasing
and decreasing patterns). Finally, we extend the preimage algorithm for a stack to preimages of certain mesh pattern classes, enabling us to give a fully automatic proof of the description of West-$3$-stack-sortable permutations,
first done in~\cite{W3}.

The algorithms treated here have been implemented in the computer algebra system Sage and are available at
\href{ham}{http://staff.ru.is/henningu/programs/pattalgos2012/pattalgos2012.html}.

\section{Learning mesh patterns}
\label{sec:grim}

Given a set of permutations $A$, we call a set of patterns $b$ a \emph{base} for $A$ if
$A = \Av(b)$.
There is a well-known algorithm that can find bases consisting only of classical patterns, which proceeds as follows if the input is the set of permutations below.
\begin{align*}
&1,
12,
21,
123,
132,
213,
312,
321,
1234,
1243,
1324,
1423,
1432,
2134,
2143,
3124,
3214,
4123,\\
&4132,
4213,
4321.
\end{align*}
The first missing permutation is $231$ so we add it to our potential base $b = \{231\}$.
When we reach the permutations of length $4$ we see that $1324$, $2314$,
$2341$, $2413$, $2431$, $3142$, $3241$, $3412$, $3421$, $4231$ and $4312$ are all missing.
All of these, except the last, contain the pattern in the base so we should expect them to be missing.
The last one, the permutation $4312$, does not contain the pattern $231$ so we extend the base,
$b = \{231, 4312\}$. If the input had permutations of length $5$ we would continue in the
same manner: checking whether the missing permutations contain a previously forbidden pattern, and
if not, extend the base by adding new classical patterns.

As the next example shows, we must also make sure that no permutation in the input contains a
previously forbidden pattern. Consider the West-$2$-stack-sortable permutations~\cite{W90}
\begin{align*}
&1,
12,
21,
123,
132,
213,
231,
312,
321,
1234,
1243,
1324,
1342,
1423,
1432,
2134,
2143,
2314,\\
&2413,
2431,
3124,
3142,
3214,
3412,
3421,
4123,
4132,
4213,
4231,
4312,
4321, \dotsc,
35241, \dotsc.
\end{align*}
The first missing permutations are $2341$ and $3241$ so $b = \{2341,3241\}$.
There are many missing permutations of length $5$ and all of them
are consequences of $2341$ being forbidden, e.g., $34152$ is not in the input since the subword
$3452$ is an occurrence of $2341$. But one of the input permutations is $35241$, which
contains the previously forbidden pattern $3241$. It seems that the presence of the
$5$ inside the occurrence of $3241$ in $35241$ is important. Here we need the notion of
mesh patterns~\cite{BC}, which allow us to forbid letters from occupying certain regions in
a pattern.
We must find a shading $R$ such that the mesh pattern $(3241,R)$ is contained in the permutation $3241$ but avoided
by $35241$. In this case it is easy to guess $R = \{(1,4)\}$ (the top-most square between
$3$ and $2$), which in fact is the correct choice: the base $b = \{2341, (3241,\{(1,4)\}\}$ exactly describes the West-$2$-stack-sortable permutations.

Now consider a more difficult input,
\begin{equation} \label{eq:difficult2}
1,21,321,2341,4123,4321.
\end{equation}
When we see that the permutation $12$ is missing we add the mesh pattern $(12,\emptyset)$ to our base. The only
permutation of length $3$ in the input is $321$, and this is consistent with our current
base. We see that every permutation of length four is forbidden up to the permutation $2341$, which
makes sense since each of those contains $(12,\emptyset)$.
But $2341$ also contains $12$,
in fact it contains several occurrences of it: $23$, $24$ and $34$. These occurrences tell us that the following mesh patterns are actually allowed.
\begin{equation} \label{eq:length2allowed}
\mpattern{scale=\pattdispscale}{ 2 }{ 1/1, 2/2 }{0/0,0/1,0/2,1/0,1/1,1/2,2/1} \qquad
\mpattern{scale=\pattdispscale}{ 2 }{ 1/1, 2/2 }{0/0,0/1,0/2,1/0,1/2,2/1,2/2} \qquad
\mpattern{scale=\pattdispscale}{ 2 }{ 1/1, 2/2 }{0/1,0/2,1/0,1/1,1/2,2/1,2/2}
\end{equation}
It is tempting to guess that what is actually forbidden is $\mpattern{scale=\patttextscale}{2}{ 1/1, 2/2 }{2/0}$. We
are forced to revisit all the permutations that are not in the input to check whether they
contain our modified forbidden pattern. This fails for the permutation $231$. We must
add something to our base, but what? Based on the mesh patterns in \eqref{eq:length2allowed}
we could add $\mpattern{scale=\patttextscale}{2}{ 1/1, 2/2 }{0/0,1/1,2/2}$, which is contained in $231$ but not
contained in any permutation in the input we have looked at up to now. From this small example it
should be clear that it quickly becomes difficult to go back and forth in the input, modifying the currently
forbidden patterns, making sure that they are not contained in any permutation in the input, while still
being contained in the permutations not in the input. We need a more unified approach. But
first, if the reader is curious the permutations in \eqref{eq:difficult2} are the permutations of length at most $4$ in the set
\begin{equation}
\label{eq:av12}
\Av \left( \mpattern{scale=\pattdispscale}{ 2 }{ 1/1, 2/2 }{0/0,1/1,2/2},\mpattern{scale=\pattdispscale}{ 2 }{ 1/1, 2/2 }{0/2,1/1,2/0} \right).
\end{equation}

\subsection{The $\bisc$ algorithm}
In the last example a mesh pattern with the underlying classical pattern $12$ first appeared in a permutation of length $4$.
This motivates the first step in the high-level description of $\bisc$:
\begin{enumerate}[(1)]
\item \label{grimsketch1} Search the permutations in the input and record which mesh patterns are allowed
\item \label{grimsketch2} Infer the forbidden patterns from the allowed patterns found in step~\eqref{grimsketch1}
\end{enumerate}
More precisely, step~\eqref{grimsketch1} is accomplished with Algorithm~\ref{algo:stein}.
\begin{algorithm}[htdp]
  \DontPrintSemicolon
  \KwIn{$A$, a finite set of permutations; and $m$, an upper bound on length of the patterns
to search for}
  \KwOut{A list, $S$, consisting of $(p,\sh_p)$ for all classical patterns of length at most $m$}
  Initialize $S = \{ (p,\emptyset) \colon p \textrm{ classical pattern of length at most } m\}$\\
  \For {$\pi \in A$}
  {
    \For{$s \in \subwords_{\leq m}(\pi)$} 
    { \nllabel{algo:steinsubw}
      Let $p = \fl(s)$\; \nllabel{algo:steinfl}
      Let $R$ be the maximal shading of $p$ for the occurrence $s$ in $\pi$\; \nllabel{algo:steinshmax}
      \If {$R \nsubseteq T$ for all shadings $T \in \sh_p$}
      { \nllabel{algo:steinsh}
        Add $R$ to $\sh_p$\;
      }
    }
  }
  \caption{$\mine$}\label{algo:stein}
\end{algorithm}

\noindent
In line~\ref{algo:steinsubw} $\subwords_{\leq m}(\pi)$ is the set of (not necessarily consecutive)
subwords of length at most $m$ in the permutation $\pi$, e.g., $\subwords_{\leq 2}(1324) =
\{13,12,14,32,34,24,1,3,2,4\}$.
In line~\ref{algo:steinshmax} we let $R$ be the maximal shading that can be applied to the classical
pattern $p$ while ensuring that $s$ is still an occurrence of $(p,R)$, see Fig.~\ref{fig:mineEx}.

\begin{figure}[htdp]
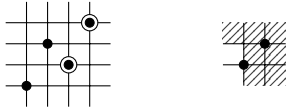

\begin{center}
\mpatternwwo{scale=\pattdispscale}{ 4 }{ 1/1, 2/3, 3/2, 4/4 }{}{3/2, 4/4}{} \qquad \mpattern{scale=\pattdispscale}{ 2 }{ 1/1, 2/2 }{0/2, 1/0, 1/1, 1/2, 2/0, 2/1, 2/2 }
\caption{Applying the maximal shading to the classical pattern $12$ so that $s = 24$ is still an occurrence in $1324$}
\label{fig:mineEx}
\end{center}
\end{figure}

\noindent
In line~\ref{algo:steinsh}, $\sh_p$ is a set we maintain of maximal shadings of $p$. For example
if we see that
$
\mpattern{scale=\patttextscale}{ 2 }{ 1/1, 2/2 }{0/0,1/1,2/2}
$
is an allowed shading, and then later come across
$
\mpattern{scale=\patttextscale}{ 2 }{ 1/1, 2/2 }{0/0,1/1,2/2,0/2,2/0}
$
we will only store the second shading, since the first one has now become redundant.

\begin{lemma} \label{lem:R'}
If $(p,R)$ is any mesh pattern, of length not longer than $m$, that occurs in a permutation in $A$
then there exists a $(p,R')$ in the output of $\mine(A,m)$ such that $R \subseteq R'$.
\end{lemma}

\begin{lemma} \label{lemma:steincl}
If a set of permutations $A$ is defined by the avoidance of a (possibly infinite)
list of classical patterns $P$ and $(q,\sh_q)$ with $\sh_q \neq \emptyset$ is in the output of $\mine(A_{\leq n},m)$
then $q$ is not in the list $P$ and $\sh_q$ only contains the full shading of $q$.
\end{lemma}

\begin{proof}
If $\sh_q$ does not contain the complete shading then $q$ is not in $A$, but appears as a classical
pattern in some larger permutation $\pi$ in $A$. Since $q$ is not in $A$ it must contain one
of the forbidden patterns $p$ in $P$, which would imply that $p$ also occurs in $\pi$, contradicting
the fact that $\pi$ is in $A$.
\end{proof}

Step~\eqref{grimsketch2} in the high-level description is implemented in Algorithm~\ref{algo:grim}, where we generate the forbidden patterns from the allowed patterns in the output of $\mine$.
\begin{algorithm}[htdp]
  \DontPrintSemicolon
  \KwIn{$S$, a list of classical patterns along with shadings $\sh_p$}
  \KwOut{A list consisting of $(p,\forb_p)$ where $p$ is a classical pattern and $\forb_p$ is a set of minimal forbidden shadings}
  \For {$(p,\sh_p) \in S$}
  {
    Let $\forb_p$ be the minimal shadings of $p$ that are not contained in any member of $\sh_p$\; \nllabel{algo:grimforb}
    \For{$R \in \forb_p$}
    {
      \If{$R$ is a consequence of some shading in $\forb_q$ for a pattern $q$ contained in $p$}
      {\nllabel{algo:grimcons1}
        Remove $R$ from $\forb_p$\; \nllabel{algo:grimcons2}
      }
    }
  }
  \caption{$\stein$} \label{algo:grim}
\end{algorithm}

\noindent
To explain lines~\ref{algo:grimcons1}--\ref{algo:grimcons2} in the algorithm assume $p = 1243$ and the shading, $R$, on the left in Fig.~\ref{fig:forbEx} is one of the shadings
in $\forb_p$ on line~\ref{algo:grimforb}.
Assume also that earlier in the outer-most for-loop we generated $q = 12$ with the shading, $R'$
on the right in Fig.~\ref{fig:forbEx}. Then the shading $R$ is removed from $\forb_p$ on line~\ref{algo:grimcons2}
since any permutation containing $(p,R)$, also contains $(q,R')$.
\begin{figure}[htdp]
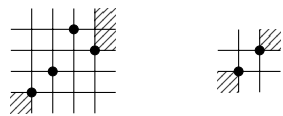

\begin{center}
\mpattern{scale=\pattdispscale}{ 4 }{ 1/1, 2/2, 3/4, 4/3 }{0/0, 4/3, 4/4 }\qquad
\mpattern{scale=\pattdispscale}{ 2 }{ 1/1, 2/2 }{0/0, 2/2 }
\caption{The mesh patterns $(p,R)$ and $(q,R')$}
\label{fig:forbEx}
\end{center}
\end{figure}

We now define $\bisc(A,m) = \stein(\mine(A,m))$.

\begin{theorem} \label{thm:biscMainThm}
Let $A$ be any set of permutations. Then for all positive integers $N \geq n$ and $m$
 \[
 A_{\leq n} \subseteq \Av(\bisc(A_{\leq N},m))_{\leq n}.
 \]
Furthermore, if the set $A$ is defined in terms of a finite list of patterns (classical or not),
 whose longest pattern has length $k$, and if $N \geq n \geq k$, $m \geq k$ then there is equality
 between the two sets above.
\end{theorem}

\begin{proof}
To prove the subset relation let $\pi$ be a permutation in $A_{\leq n}$ and $(p,R)$ be a mesh pattern of length at most $m$
contained in $\pi$. By Lemma~\ref{lem:R'} there is a pattern $(p,R')$ in the output from
$\mine(A_{\leq N},m)$ such that $R \subseteq R'$.
This implies that $R$ is not one of the shadings in $\forb_p$ and therefore $(p,R)$
is not in the output of $\bisc(A_{\leq n},m)$.
To prove equality of the sets let $\pi$ be a permutation that is not in the set on the right because it contains
a mesh pattern $(p,R)$ from the list defining $A$. If $p$ never occurs as a classical pattern in a
permutation in $A_{\leq N}$ then $(p,\emptyset)$ is in the output of $\mine(A_{\leq N},m)$ which implies
that $(p,\emptyset)$ is in the output of $\bisc(A_{\leq N},m)$. This implies that $\pi$ is not in the set
on the right. If, however, $p$ occurs in some permutations in $A_{\leq N}$ then every occurrence of a mesh
pattern $(p,R')$ must satisfy $R \nsubseteq R'$. This implies that when $\forb_p$ is created in
line~\ref{algo:grimforb} in Algorithm~\ref{algo:grim} it contains $R$ or a non-empty subset, $R''$, of it.
We assume without loss of generality that $R''$ is not removed due to redundancy in line~\ref{algo:grimcons2}.
This implies that $(p,R'')$ is in the output of $\bisc(A_{\leq N},m)$ so $\pi$ is not in the set on the right.
\end{proof}

If the input $A$ to $\bisc$ is defined in terms of classical patterns we can strengthen the previous theorem.

\begin{theorem}
If a set of permutations $A$ is defined by the avoidance of a (possibly infinite)
list $P$ of classical patterns then the output of $\bisc(A_{\leq n},m)$ for any $n,m$ will consist only
of classical patterns.
Furthermore, if the longest patterns in the list $P$ have length $k$, then
$A = \Av(\bisc(A_{\leq n},m))$ for any $n,m \geq k$.
\end{theorem}
\begin{proof}
This follows from Lemma~\ref{lemma:steincl}, see~\cite{HU} for the details.
\end{proof}

Some properties require us to look at very large permutations to discover patterns. For
example one must look at permutations of length $8$ in the set in~\eqref{eq:av12} above
to see that the mesh pattern $\mpattern{scale=\patttextscale}{ 2 }{ 1/1, 2/2 }{0/0,0/2,1/0,2/0,2/1,2/2}$ is allowed.
But as examples in the next section show it often suffices to look at permutations of length $m+1$ when
searching for mesh patterns of length $m$.

For some input sets $A$ there can be redundancy in the output $\bisc(A,m)$, in the sense that
some patterns can be omitted without changing $\Av(\bisc(A,m))$. We propose three ways to fix this
in~\cite{U}. This is never a problem in the following applications so we omit further details.

\subsection{Some applications of $\bisc$}
\label{subsec:appl}

\paragraph{Forbidden shapes in Young tableaux.}
There is a well-known bijection, called the Robinson-Schensted-Knuth-correspondence (RSK) between permutations of length $n$
and pairs of Young tableaux of the same shape. Let $A$ be the set of permutations whose image tableaux are hook-shaped. The output of $\bisc(A_{\leq 5},4)$ is the following four patterns.
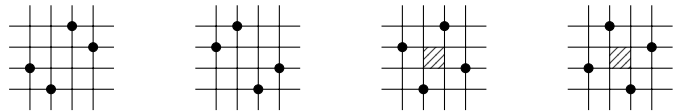
\[
\mpattern{scale=\pattdispscale}{ 4 }{ 1/2, 2/1, 3/4, 4/3 }{} \qquad
\mpattern{scale=\pattdispscale}{ 4 }{ 1/3, 2/4, 3/1, 4/2 }{} \qquad
\mpattern{scale=\pattdispscale}{ 4 }{ 1/3, 2/1, 3/4, 4/2 }{2/2} \qquad
\mpattern{scale=\pattdispscale}{ 4 }{ 1/2, 2/4, 3/1, 4/3 }{2/2}
\]
Here we have rediscovered an observation made by Atkinson~\cite[Proof of Lemma 9]{A}
that the avoiders of the four patterns above are exactly the permutations whose tableaux are
hook-shaped.
We can view the property of being a hook-shaped tableau as not containing
the tableaux shape $(2,2)$ (two rows with two boxes each). It is natural to ask if we can
in general describe the permutations corresponding to tableaux not containing a particular
shape $\lambda$. E.g., let $A$ be the permutations whose tableaux avoid $\lambda = (3,2)$.
The output of $\bisc(A_{\leq 6},5)$ is $25$ mesh patterns that can be simplified to give the
following.
\begin{proposition}
The permutations whose tableaux under the RSK-correspondence avoid the shape $(3,2)$
are precisely the avoiders of the marked mesh patterns (\cite[Definition 4.5]{Unific}) below.
\[
\mmpatterna{scale=\pattdispscale}{ 4 }{ 1/2, 2/1, 3/4, 4/3 }{}{0/0/1/1/{},4/4/5/5/{}} \qquad
\mmpatternb{scale=\pattdispscale}{ 4 }{ 1/3, 2/4, 3/1, 4/2 }{}{1/3/2/4/1,3/1/4/2/{}} \qquad
\mmpattern{scale=\pattdispscale}{ 4 }{ 1/3, 2/1, 3/4, 4/2 }{2/2}{0/0/2/1/1,3/4/5/5/{}} \qquad
\mmpattern{scale=\pattdispscale}{ 4 }{ 1/2, 2/4, 3/1, 4/3 }{2/2}{0/0/1/2/1,4/3/5/5/{}}
\]
\end{proposition}
The meaning of a marked region is that it should contain at least one point, and therefore
the first pattern above corresponds to $9$ classical patterns.

Crites et al.~\cite{CPW} showed that if a permutation $\pi$ contains a separable classical pattern
$p$ then the tableaux shape of $p$ is contained in the tableaux shape of $\pi$. Atkinson's observation
and the proposition above suggest a more general result holds if one considers mesh
patterns instead of classical patterns.

\paragraph{Forbidden tree patterns.}
Pudwell et al.~\cite{LaraTalk} used the bijection between binary trees and $231$-avoiding permutations and
the $\bisc$ algorithm to discover a correspondence between tree patterns and
mesh patterns. This can be used to generate Catalan many Wilf-equivalent sets of the form $\Av(231,p)$
where $p$ is a mesh pattern. The tree in Fig.~\ref{fig:bintree} gives the Wilf-equivalence between the
sets $\Av(231,654321)$ and $\Av(231,(126345,\{(1,6),(4,5),(4,6)\})$.

\begin{figure}[htb!]
\begin{center}
\begin{tikzpicture}[scale = 0.4, place/.style = {circle,draw = black!50,fill = gray!20,thick,inner sep=2pt}, auto]

 \node [place] (1) at (0,2)     {};
 \node [place] (2) at (2,1)      {};
 \node [place] (3) at (1,0)     {};
 \node [place] (4) at (0,-1)       {};
 \node [place] (5) at (-2,1)       {};
 \node [place] (6) at (-1,0)     {};
 \node [place] (7) at (-3,0)     {};
 \node [place] (8) at (-4,-1)     {};
 \node [place] (9) at (-2,-1)     {};
 \node [place] (10) at (3,0)     {};
 \node [place] (11) at (2,-1)     {};
 \node [place] (12) at (1,-2)     {};
 \node [place] (13) at (-1,-2)     {};
  
 \draw [-,semithick]      (1) to (2);
 \draw [-,semithick]      (2) to (3);
 \draw [-,semithick]      (3) to (4);
 \draw [-,semithick]      (1) to (5);
 \draw [-,semithick]      (5) to (6);
 \draw [-,semithick]      (5) to (7);
 \draw [-,semithick]      (7) to (8);
 \draw [-,semithick]      (7) to (9);
 \draw [-,semithick]      (2) to (10);
 \draw [-,semithick]      (3) to (11);
 \draw [-,semithick]      (4) to (12);
 \draw [-,semithick]      (4) to (13);
 
 \draw [<->, thick] (4,0) to (5,0);
 
\node [] (x) at (5.5,0) [label = right:$\mpattern{scale=\pattdispscale}{ 6 }{ 1/1, 2/2, 3/6, 4/3, 5/4, 6/5 }{1/6,4/5,4/6}$] {};

\end{tikzpicture}
\caption{A binary tree and its corresponding mesh pattern, see~\cite{LaraTalk} }
\label{fig:bintree}
\end{center}
\end{figure}
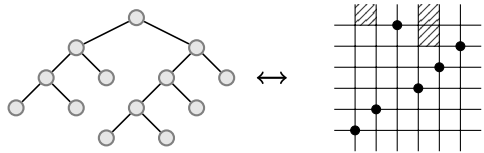

\paragraph{$1$-quicksortable permutations.}
Consider the following sorting method $\Pi$: If the input permutation
is empty, return the empty permutation. If the input permutation $\pi$ is not empty
and contains strong fixed points,
let $x$ be the right-most such point, write $\pi = \alpha x \beta$ and recursively apply the
method to $\alpha$ and $\beta$. If $\pi$ does not contain a strong fixed point we move every
letter in $\pi$ that is smaller than the first letter to the start of the permutation.
This operator was proposed by Claesson~\cite{AC} as a single pass
of quicksort~\cite{Hoare}. Let $A$ be the set of permutations that are sortable in
at most one pass under this operator. The output of $\bisc(A_{\leq 5},4)$ is the following.
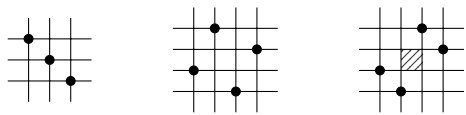
\[
\mpattern{scale=\pattdispscale}{ 3 }{ 1/3, 2/2, 3/1 }{} \qquad
\mpattern{scale=\pattdispscale}{ 4 }{ 1/2, 2/4, 3/1, 4/3 }{} \qquad
\mpattern{scale=\pattdispscale}{ 4 }{ 1/2, 2/1, 3/4, 4/3 }{2/2}
\]
Stuart Hannah~\cite{H} verified that the sortable permutations are precisely the avoiders
of these three patterns and, furthermore, that they are in bijection
with words on the alphabet $(a + (bb\star)(cc\star))\star$ which are counted by the
sequence~\cite[A034943]{OEIS}.
\newline

The next section presents algorithms that can prove a complete description of
the preimage of pattern classes under two sorting devices: a stack of depth $d$ and
a queue. More devices are treated in~\cite{HU}.

\section{Sorting algorithms and preimages}
\label{sec:sortingalgo}

\subsection{Sorting with a stack of any depth}
  A \emph{stack} is a list with the restriction that elements can only be added and removed from one end of the list. We call this end the \emph{top} of the stack. The act of adding an element to a stack is called \emph{pushing}, and the act of removing is called \emph{popping}. The \emph{depth} of a stack is the number of elements it can hold along with one extra space used for passing elements through the stack. 
  
  Given an input permutation $\pi$, let $S(\pi)$ be the permutation obtained by the following procedure~\cite{K}: 
  \begin{inparaenum}
    \item If the stack is empty or the topmost element is larger than the first element of the input, $s$, push $s$ onto the stack.
    \item Otherwise, pop elements from the stack into the output permutation, until $s$ can be pushed onto the stack.
    \item Repeat this process until the input is empty.
    \item Empty the stack into the output permutation.
  \end{inparaenum}
  
  The same procedure can also be used with stacks of limited depth. We let $S_d(\pi)$ denote the permutation obtained by applying the procedure, with a stack of depth $d$, to a permutation $\pi$. It can be easily proven that for a permutation $\pi = \alpha n \beta$, where $n$ is the largest element of $\pi$, that
  \[
    S_d(\pi) = S_d(\alpha n \beta) = 
    \begin{cases}
      S_d(\alpha) S_{d-1}(\beta) n &\text{if $d > 1$,} \\
      \pi &\text{if $d = 1$.}
    \end{cases}
  \]
  We call $S_d$ the \emph{stack-sort operator with a stack of depth $d$}. If the result of applying $S_d$ to a permutation $\pi$ is the identity permutation, we say that $\pi$ is \emph{stack-sortable with a stack of depth $d$}.
  
  Note that $S_\infty = S$ is the stack-sort operator (of unlimited depth), and $S_2 = B$ is the bubble sort operator (see e.g.~\cite{AABCD}). If we apply $S_3$ to the permutation $45321$, then $4$ is pushed onto the stack, but immediately popped by the $5$. Now we have $4$ in the output and $5$ on the stack. The $3$ is then pushed onto the stack and the stack is now full. Finally $2$ and $1$ bypass the stack, and the stack is subsequently emptied into the output. Thus $S_3(45321) = 42135$.
  
  Algorithm 1 in~\cite{CU} gives a description of the preimage of any set defined by the avoidance of classical patterns for $S$. More precisely, given a classical pattern $p$, the algorithm outputs a set $M$ of marked mesh patterns such that 
  \[
    \Av(M) = \{ \pi : S(\pi) \in \Av(p) \}.
  \]
  
  The algorithm proceeds in two steps. In the first step classical patterns, called \emph{candidates}, are generated. In the second step, shadings and markings are added to the candidates, to produce the patterns that describe the preimage. Note that it might be impossible to add shadings or markings to some of the candidates produced by the first step. We will now give a generalization of this algorithm to handle the preimage of a stack of any depth.
  
  The analog of~\cite[Proposition 4.1]{CU}, which generates classical pattern candidates, is the following.
  \begin{proposition} \label{prop:cand}
  Let $p = \alpha n \beta$ be a permutation of a finite set of integers where $n$ is the largest element of $p$ and $\alpha = a_1 a_2 \dotsm a_i$. Then
  \[
    \cand_d(p) = 
    \begin{cases}
      \bigcup_{j = 0}^i \left\{ \gamma n \delta : \gamma \in \cand_d(a_1 a_2 \dotsm a_j) , \delta \in \cand_{d-1}(a_{j+1} \dotsm a_i \beta) \right\} &\text{if $d > 1$,} \\
      p &\text{if $d = 1$.}
    \end{cases}
  \]
  contains all classical patterns that can become $p$ after one pass of $S_d$.
  \end{proposition}

  To complete the algorithm, we need the notion of a \emph{decorated pattern}~\cite[Definition 2.2]{W3}. As was hinted at in the paper, we need a slight modification of these patterns, where some regions are \emph{required} to contain a pattern. We therefore view a decorated pattern as a $5$-tuple $(p, S, M, D, C)$, where $p$ is the underlying classical pattern, $S$ is the set of shaded squares, $M$ is the set of markings, $D$ is the set of avoidance decorations, and $C$ is the set of containment decorations.

  It can be easily seen that if $\sigma = S_d(\pi)$, then $\inv(\sigma) \subseteq \inv(\pi)$, i.e., when the stack-sort operator is applied to $\pi$, each inversion in $\pi$ either becomes an non-inversion in the output $\sigma$, or it stays an inversion. Non-inversions in $\pi$ must also be non-inversions in $\sigma$.
  
  The second step of the algorithm proceeds as follows. Suppose we have a candidate $\lambda \in \cand_d(p)$
  and a permutation $\pi$ containing $\lambda$.
  An inversion in $\lambda$, corresponding to the elements $a$ and $b$ in $\pi$, will remain an inversion
  after a pass through the stack if either of the two cases below applies.
  
  \begin{enumerate}[(1)]
  \item There is an element $c$, larger than $a$, that appears between $a$ and $b$ in $\pi$. In this case $a$ would be pushed onto the stack, and subsequently popped off by $c$, before $b$ would be pushed onto the stack. The relative order of $a$ and $b$ would therefore remain the same, after applying the stack-sort operator. This corresponds to the pattern $C_1$ in Tab.~\ref{tab:sk-preimages}.
  
  \item The stack is full of elements larger than $a$, when $a$ appears in the input. In this case $a$ bypasses the stack and appears before $b$ in the output. This happens precisely when there is a sequence of decreasing elements, all larger than $a$, that appear in $\pi$ before $a$. This sequence also has the property that all its elements are left-to-right maxima in $\pi$. Otherwise, at least one element of the sequence would be popped of the stack before $a$ appears. For a stack of depth $d+1$, this property is described by the pattern
  \[
    V_d = ( d (d-1) \dotsm 1, \{ (i,j) \in \dbrac{0,d} \times \dbrac{0,d} : (d - i) > j \} ).
  \]
  This case corresponds to the pattern $C_2$ in Tab.~\ref{tab:sk-preimages}, where the region decorated with $V_d$ must contain the pattern $V_d$.
  \end{enumerate}
  If neither of the two cases applies to an inversion in $\lambda$, it must become a non-inversion in the output. This corresponds to the pattern $C_3$ in Tab.~\ref{tab:sk-preimages}, where the shaded region decorated with $V_d$ must avoid the pattern $V_d$. This process is formalized in Algorithm \ref{algo:stackpreim}, which extends \cite[Algorithm 1]{CU}.
  
  \newcommand{\vthree}{
    \scalebox{.5}{$V_d$}
  }
  \begin{table}[hb!]
    \centering
    \begin{tabular}{p{4cm}ll}
      \toprule
      \textbf{Description} & \textbf{Candidates} & \textbf{Output}\\
      \midrule
      Inversions that stay inverted &
      $C_1 = \descpatt{scale=\patttablecale}{2}{2}{1/2,2/1}{}{1/2/2/3/1}{}{}, \quad
      C_2 = \descpatt{scale=\patttablecale}{2}{2}{1/2,2/1}{1/2/2/3}{}{}{0/2/1/3/{\vthree}}$
      & 
      \descpatt{scale=\patttablecale}{2}{2}{1/2,2/1}{}{}{}{} \\
      Inversions that become non-inversions &
      $C_3 = \descpatt{scale=\patttablecale}{2}{2}{1/2,2/1}{1/2/2/3}{}{0/2/1/3/{\vthree}}{}$
      &
      \descpatt{scale=\patttablecale}{2}{2}{1/1,2/2}{}{}{}{} \\
      \bottomrule
    \end{tabular}
    \caption{The decorations, markings and shadings needed for a stack of depth $d$}
    \label{tab:sk-preimages}
  \end{table}
  
  \begin{algorithm}[ht!]
    \DontPrintSemicolon
    \KwIn  {The depth $d$ of the stack; a pattern $p$; and $\lambda \in \cand_d(p)$}
    \KwOut {A (possibly empty) set of decorated patterns $T$}
    Let $T = \{ (\lambda, \emptyset, \emptyset, \emptyset, \emptyset) \}$ and let $n$ be the length of $\lambda$\;
    \For {$(i,j) \in \inv(\lambda)$}
    {
      Let $R_1 = \dbrac{\lambda^{-1}(i),\lambda^{-1}(j) - 1} \times \dbrac{i,n}$ and $R_2 = \dbrac{0,\lambda^{-1}(i) - 1} \times \dbrac{i,n}$\;
      \For {$r = (\lambda, S, M, D, C) \in T$}
      {
        Remove $r$ from $T$\;
        \uIf {$(i,j) \in \inv(p)$}
        {
          \If{$R_1 \not\subseteq S$}
          {
            \lIf{An element of $\lambda$ is contained in $R_1$}
            {
              Add $r$ to $T$\;
            }
            \lElse
            {
              Add $(\lambda, S, M \cup \{ (R_1 \setminus S, 1) \}, D, C)$ to $T$\;
            }
          }
          \If{No element of $\lambda$, no marking in $M$, or decoration in $C$, is contained in $R_1 \cup S$ and $R_2$ is not contained in any decoration in $D$}
          {
            Add $(\lambda, S \cup R_1, M, D, C \cup \{ (R_2,V_{d-1})) \}$ to $T$\;
          }
        }
        \ElseIf{No element of $\lambda$, no marking in $M$, or decoration in $C$, is contained in $R_1 \cup S$ and no decoration from $C$ is contained in $R_2$}
        {
          Add $(\lambda, S \cup R_1, M, D \cup \{ (R_2,V_{d-1}) \}, C)$ to $T$\;
        }
      }
    }
    \caption{An algorithm describing how a particular candidate must be decorated}
    \label{algo:stackpreim}
  \end{algorithm}


  By choosing $p=21$ and applying Algorithm~\ref{algo:stackpreim} we easily arrive at the following proposition.

  \begin{proposition} \label{prop:depthd}
    Stack-sortable permutations, with a stack of depth $d$, are precisely the avoiders of $231$ and $(d+1)d \dotsm 2 1$.
  \end{proposition}%
  Note that Goodrich et al.~\cite{LaraTalk2} independently discovered this proposition.
  To describe the permutations sortable by two passes through a stack of depth $d$, one would need to describe the preimage of the avoiders of $p=231$. In this case $\cand_d(p) = \{231, 321\}$ (if $d > 1$). Algorithm~\ref{algo:stackpreim} then gives
  \[
    \descpatt{scale=\pattdispscale}{3}{3}{1/3,2/2,3/1}{1/3/2/4}{2/3/3/4/1}{0/3/1/4/\vthree}{} \qquad
    \descpatt{scale=\pattdispscale}{3}{3}{1/2,2/3,3/1}{}{2/3/3/4/1}{}{} \qquad
    \descpatt{scale=\pattdispscale}{3}{3}{1/2,2/3,3/1}{2/3/3/4}{}{}{0/3/2/4/\vthree}
  \]
  To complete the description, one also needs to determine the preimage of $(d+1)d \dotsm 1$, see~\cite{HU}.

\subsection{Sorting with a queue}
  A \emph{queue} is a list with the restrictions that elements can only be added to one end of the list, called the \emph{front} and removed from the other end of the list, called the \emph{back}. The act of adding an element to a queue is called \emph{enqueuing}, and the act of removing is called \emph{dequeuing}. 
  
  Given an input permutation $\pi$, let $Q(\pi)$ be the permutation obtained by the following procedure~\cite{K}: 
  \begin{inparaenum}
    \item If the queue is empty or the last element is smaller than the first element of the input, $s$, enqueue $s$.
    \item Otherwise, dequeue elements from the queue into the output permutation until the front element of the queue is larger than $s$, or the queue is empty. Add $s$ to the output permutation.
    \item Repeat this process until the input is empty.
    \item Empty the queue into the output permutation.
  \end{inparaenum}

  We will now sketch an algorithm for describing the preimage of any set defined by the avoidance of classical patterns. More precisely, given a classical pattern $p$, the algorithm outputs a set $M$ of decorated patterns such that $\Av(M) = \{ \pi : Q(\pi) \in \Av(p) \}$. In~\cite{HU} we give a proposition analogous to Proposition~\ref{prop:cand}, but here we consider the candidates $\{ \lambda : \inv( p ) \subseteq \inv( \lambda ) \}$ for simplicity. The preimage algorithm for $Q$ proceeds similar to the algorithm for a stack of depth $d$, except when we consider how inversions change, we use the patterns shown in Tab. \ref{tab:queuepreim}.

  \newcommand{\twoone}{\mpatternpic{.08}{2}{1/2, 2/1}{}}
  \begin{table}[htb!]
    \centering
    \begin{tabular}{ll}
      \toprule
      \textbf{Candidates} & \textbf{Output} \\
      \midrule
      $D_1 = \descpatt{scale=\patttablecale}{2}{2}{1/2,2/1}{}{0/2/1/3/1}{}{}, \quad
      D_2 = \descpatt{scale=\patttablecale}{2}{2}{1/2,2/1}{0/2/1/3}{}{}{1/2/2/3/{\twoone}}$
      &
      \descpatt{scale=\patttablecale}{2}{2}{1/2,2/1}{}{}{}{} \\
      
      $D_3 = \descpatt{scale=\patttablecale}{2}{2}{1/2,2/1}{0/2/1/3}{}{1/2/2/3/{\twoone}}{}$
      &
      \descpatt{scale=\patttablecale}{2}{2}{1/1,2/2}{}{}{}{} \\
      \bottomrule
    \end{tabular}
    \caption{The decorations, markings and shadings needed for a queue}
    \label{tab:queuepreim}
  \end{table}

  Let $r$ denote the reverse of a permutation, and $c$ denote the complement of a permutation. By using the algorithm for a queue and a stack~\cite{CU}, the following theorem is proved.

  \begin{theorem}
  A permutation $\pi$ avoids $4312$ if and only if $(S \circ r \circ c \circ Q)(\pi)$ is sorted.
  \end{theorem}%
  Since all the maps in the composition $S \circ r \circ c \circ Q$ are linear time operators and
  it takes linear time to check whether the output is sorted, this provides a linear time algorithm
  for checking for the avoidance of $4312$. Such linear time algorithms have only been known for
  patterns of length at most $3$ as well as the increasing and decreasing patterns of any length
  ($1 2 \dotsm k$ and $k \dotsm 2 1$). Also see Albert et al.~\cite{nlogn} for algorithms with
  running time $n\log{n}$ for patterns of length $4$.

  %
  %
%
      %

  One can similarly find preimage algorithms for a pop-stack, insertion sort and pancake sort, see~\cite{HU}. In all of the algorithms above, we have only considered preimages of classical pattern classes. However, in~\cite{HU}, we extend the preimage algorithm for a stack to handle certain mesh pattern classes. This allows us to give an automatic proof of the description of the West-$3$-stack-sortable permutations.

\section*{Acknowledgments}
\label{sec:ack}
The second author would like to thank Sara Billey for asking the question that
led to the development of the $\bisc$ algorithm as well as helpful discussions
on how to implent it. He would also like to thank Anders Claesson for suggesting that
the forbidden patterns might somehow be inferred from the allowed patterns. He
would finally like to thank Einar Steingr{\'i}msson for many helpful discussions
on permutation patterns. The name of the algorithm is derived from the names of these
three people.

\bibliographystyle{alpha}
\bibliography{magulf-refs}
\label{sec:biblio}

\end{document}

%% file: patternmacros.tex
\newcommand{\shadetheboxes}[1]{
	\foreach \x/\y in {#1}
      	\fill[pattern color = black!75, pattern=north east lines] (\x,\y) rectangle +(1,1);
	}

\newcommand{\shadetherects}[1]{
	\foreach \x/\y/\xt/\yt in {#1}
      	\fill[pattern color = black!75, pattern=north east lines] (\x,\y) rectangle (\xt,\yt);
	}

\newcommand{\drawthegrid}[1]{
	\draw (0.01,0.01) grid (#1+0.99,#1+0.99);
	}
	
\newcommand{\drawverticallines}[3]{
	\foreach \x in {#2}
		\draw[line width=#3] (\x+0.01,0.01) -- (\x+0.01,#1+0.99);
	}

\newcommand{\drawhorizontallines}[3]{
	\foreach \y in {#2}
		\draw[line width=#3] (0.01,\y+0.01) -- (#1+0.99,\y+0.01);
	}
	
\newcommand{\drawtheclpattern}[1]{
	\foreach \x/\y in {#1}
      	\filldraw (\x,\y) circle (6pt);
	}
	
\newcommand{\drawtheclpatternwhite}[1]{
	\foreach \x/\y in {#1}
      	\draw[fill=white] (\x,\y) circle (6pt);
	}
	
\newcommand{\drawtheclpatternwhitebig}[1]{
	\foreach \x/\y in {#1}
		\draw[fill=white] (\x,\y) circle (11pt);
	}
	
\newcommand{\drawspecialbox}[1]{
	\foreach \x/\y/\z/\w/\A in {#1}
		{
       		\fill[color = white!100, opacity=1, rounded corners = 1.5pt] (\x+0.125,\y+0.125) rectangle (\z-0.125,\w-0.125);
       		\draw[color = black, rounded corners = 1.5pt] (\x+0.125,\y+0.125) rectangle (\z-0.125,\w-0.125);
       		\fill[black] (\x/2+\z/2,\y/2+\w/2) node {$\scriptstyle\A$};
       	}
    }

\newcommand{\onetwo}{\mpatternnl{scale=0.2}{2}{1/1,2/2}{}}		

\newcommand{\drawsolidshadedbox}[1]{
	\foreach \x/\y/\z/\w/\A in {#1}
		{
       		\fill[color = gray!50, opacity=1, rounded corners=1.5pt] (\x+0.125,\y+0.125) rectangle (\z-0.125,\w-0.125);
       		\draw[color = black, rounded corners=1.5pt] (\x+0.125,\y+0.125) rectangle (\z-0.125,\w-0.125);
       		\fill[black] (\x/2+\z/2,\y/2+\w/2) node {$\scriptstyle\A$};
       	}
    }

\newcommand{\drawascendingrestrictions}[1]{
	\foreach \x/\y/\z/\w in {#1}
		{
       		\fill[color = white!100, opacity=1, rounded corners=1.5pt] (\x+0.125,\y+0.125) rectangle (\z-0.125,\w-0.125);
       		\fill[pattern color = black!65, pattern=north east lines, rounded corners=1.5pt] (\x+0.125,\y+0.125) rectangle (\z-0.125,\w-0.125);
       		\draw[color = black, rounded corners=1.5pt] (\x+0.125,\y+0.125) rectangle (\z-0.125,\w-0.125);
       		\fill[black] (\x/2+\z/2,\y/2+\w/2) node {\onetwo};
       	}
    }    

\newcommand{\mpattern}[4]{										
  \raisebox{0.6ex}{
  \begin{tikzpicture}[scale=0.35, baseline=(current bounding box.center), #1]
  	\useasboundingbox (0.0,-0.1) rectangle (#2+1.4,#2+1.1);
	
    \shadetheboxes{#4}
    
    \drawthegrid{#2}
    
    \drawtheclpattern{#3}
    
  \end{tikzpicture}}
}

\newcommand{\mpatternpic}[4]{										
	\scalebox{#1}
	{  
		\begin{tikzpicture}[baseline=(current bounding box.center)]
		
			\foreach \x/\y in {#4}
				\fill[pattern color = black!65, pattern=north east lines] (\x,\y) rectangle +(1,1);

			\drawthegrid{#2}
			\drawtheclpattern{#3}
		
		\end{tikzpicture}
	}
}

\newcommand{\mpatternnl}[4]{										
  \raisebox{0.6ex}{
  \begin{tikzpicture}[scale=0.35, baseline=(current bounding box.center), #1]
  	\useasboundingbox (0.85,-0.1) rectangle (#2+1.4,#2+1.1);
	
    \shadetheboxes{#4}
    
    \drawtheclpattern{#3}
    
  \end{tikzpicture}}
}

\newcommand{\mpatternwwo}[6]{									
  \raisebox{0.6ex}{
  \begin{tikzpicture}[scale=0.35, baseline=(current bounding box.center), #1]
  \useasboundingbox (0.0,-0.1) rectangle (#2+1.4,#2+1.1);
  
    \shadetheboxes{#6}
    
    \drawthegrid{#2}
    
    \drawtheclpatternwhite{#4}
    \drawtheclpatternwhitebig{#5}
    \drawtheclpattern{#3}
    
  \end{tikzpicture}}
}

\newcommand{\mmpattern}[5]{									
  \raisebox{0.6ex}{
  \begin{tikzpicture}[scale=0.35, baseline=(current bounding box.center), #1]
  \useasboundingbox (0.0,-0.1) rectangle (#2+1.4,#2+1.1);
    
    \shadetheboxes{#4}
    
    \drawthegrid{#2}
    
    \drawspecialbox{#5}
    
    \drawtheclpattern{#3}

  \end{tikzpicture}}
}

\newcommandx{\descriptionpatt}[9][8={},9={}]
{
	\raisebox{0.6ex}{
  \begin{tikzpicture}[scale=0.35, baseline=(current bounding box.center), #1]
  	\useasboundingbox (0.0,-0.1) rectangle (#2+1.4,#3+1.1);
		
		\shadetherects{#7}

		\foreach \pos/\type in {#4}
		{
			\ifthenelse{\equal{\type}{v}}
			{
				\drawverticallines{#3}{\pos}{0.8pt}
			}
			{
				\drawhorizontallines{#2}{\pos}{0.8pt}
			}
		}

		\foreach \pos/\type in {#5}
		{
			\ifthenelse{\equal{\type}{v}}
			{
				\drawverticallines{#3}{\pos}{0.4pt}
			}
			{
				\ifthenelse{\equal{\type}{d}}
				{
					\draw[densely dashed] (\pos,0) -- (\pos,#3+1);
				}
				{
					\drawhorizontallines{#2}{\pos}{0.4pt}
				}
			}
		}
		
		\drawascendingrestrictions{#8}
		\drawspecialbox{#9}
		
		\foreach \x/\y/\type in {#6}
		{
			\ifthenelse{\equal{\type}{a}}
			{
				\draw (\x,\y) circle (6pt);
				\filldraw (\x,\y) circle (3pt);
			}
			{
				\filldraw (\x,\y) circle (5pt);
			}
		}
    
  \end{tikzpicture}}
} 

\newcommand{\descpatt}[8]
{
	\begin{tikzpicture}[scale=0.35, baseline=(current bounding box.center), #1]
		\shadetherects{#5}
		\draw (0.01,0.01) grid (#2+1-0.01,#3+1-0.01);
		
		\drawsolidshadedbox{#7}
		\drawspecialbox{#6}
		\drawspecialbox{#8}
		\foreach \x/\y/\type in {#4}
		{
			\filldraw (\x,\y) circle (5pt);
		}
  \end{tikzpicture}
}

\newcommand{\descpatte}[8]
{
	\begin{tikzpicture}[scale=0.35, baseline=(current bounding box.center), #1]
		\shadetherects{#5}
		\draw (0.01,0.01) grid (#2+1-0.01,#3+1-0.01);
		
		\drawsolidshadedbox{#7}
		\drawspecialbox{#6}
		\drawspecialbox{#8}

		\foreach \x/\y/\type in {#4}
		{
			\ifthenelse{\equal{\type}{s}}
			{
				\draw (\x,\y) circle (6pt);
				\filldraw (\x,\y) circle (3pt);
			}
			{
				\filldraw (\x,\y) circle (5pt);
			}
		}
		
  \end{tikzpicture}
}